\documentclass[12pt]{article}
\usepackage{bbm}
\usepackage{graphicx}
\graphicspath{ {./images/} }
\usepackage{wrapfig}
\usepackage[mathscr]{euscript}
\usepackage{subcaption}
\usepackage{hyperref}
\usepackage{epstopdf}
\usepackage{enumerate}
\usepackage{amsmath,amsfonts,amssymb,amsthm,epsfig,epstopdf,titling,url,array}
\usepackage{color}
\usepackage{framed}
\usepackage[utf8]{inputenc}
\usepackage[english]{babel}
\usepackage{xparse}
\usepackage{authblk}
\usepackage[margin=2.4cm]{geometry}
\usepackage{authblk}
\newtheorem{cor}{Corollary}[section]
\newtheorem{lem}{Lemma}[section]
 \newtheorem{prop}{Proposition}[section]

 \newtheorem{rem}{Remark}[section]

 \newtheorem{maintheorem}{Theorem}

 \date{}
 
\begin{document}
\title{ Chebyshev's method for exponential maps }
 
\author[1]{Subhasis Ghora 
   \footnote{subhasisghora@iiitdm.ac.in}}
\author[2]{Tarakanta Nayak 
    \footnote{tnayak@iitbbs.ac.in}}
\author[3]{Soumen Pal \footnote{soumen.pal.new@gmail.com}}
\author[2]{Pooja Phogat  \footnote{Corresponding author, poojaphogat174acad@gmail.com}}
\affil[1]{Department of Sciences and Humanities, Indian Institute of Information Technology Design and Manufacturing Kancheepuram, Chennai, India}
\affil[2]{Department of Mathematics,  
		Indian Institute of Technology Bhubaneswar, Bhubaneswar, India}
\affil[3]{Department of Mathematics, 
		Indian Institute of Technology Madras, Chennai, India}
	\date{}
\maketitle
\begin{abstract}
 It is proved that the Chebyshev's method applied to an entire function $f$ is a rational map if and only if   $f(z) = p(z) e^{q(z)}$, for some polynomials $p$ and $q$. These are referred to as rational Chebyshev maps, and their fixed points are discussed in this   article. It is seen that $\infty$ is a parabolic fixed point with multiplicity one bigger than the degree of  $q$. Considering $q(z)=p(z)^n+c$, where $p$ is a linear polynomial, $n \in \mathbb{N}$ and $c$ is a non-zero constant, we show that the Chebyshev's method applied to $ pe^q$ is affine conjugate to that applied to $z e^{z^n}$. We denote this by $C_n$. All the finite extraneous fixed points of $C_n$ are shown to be repelling. The Julia set $\mathcal{J}(C_n)$ of $C_n$ is found to be preserved under rotations of order $n$ about the origin. For each $n$, the immediate basin of $0$ is proved to be simply connected. For all $n \leq 16$, we prove that $\mathcal{J}(C_n)$ is connected.  For $n$ even, the non-existence of  Herman ring and  Siegel disk  of $C_n$ is proved. Under some additional hypothesis, the same is also proved for odd $n$.
The Newton's method applied to $ze^{z^n}$ is found to be conjugate to a polynomial, and its dynamics is also completely determined.  
\end{abstract}
\section{Introduction}
 
A root-finding method $F$ usually associates an entire function $f: \mathbb{C} \to \mathbb{C}$ to a meromorphic function $F_f$ such that every root $z_0$ of $f$ is an attracting fixed point for $F_f$, i.e., $F_f (z_0)=z_0$ and $|F_f ' (z_0)| < 1$. Starting from an initial guess $z$, the successive iterates $\{F_f^n(z)\}_{n\geq0}$ are used to approximate the roots of $f$. This naturally leads to the study of $\{F_f^n(z)\}_{n\geq0}$ for all $z$ in the extended complex plane $\widehat{\mathbb{C}}$. The study of the iteration of Newton's method applied to quadratic polynomials, initiated by Cayley in 1879~\cite{Cayley1879}, marked the beginning of Complex Dynamics. Complex Dynamics is the branch of mathematics concerned with the iteration of rational and transcendental functions. Given a meromorphic function $R:\widehat{\mathbb{C}}\to\widehat{\mathbb{C}}$, the behavior of the sequence of iterates $\{R^n\}_{n\geq0}$ leads to a decomposition of the Riemann sphere into the Fatou set (where the dynamics is stable) and the Julia set (where the dynamics is chaotic). The theory of Fatou and Julia sets has developed into a rich area of research (see for instance, \cite{Beardon_book,Milnor_book} for introductions) with deep connections to several branches of mathematics.

 The Newton's method applied to an entire function  $f$ is given by $N_f(z)=z-\frac{f(z)}{f'(z)}$. This method is quadratic convergent, i.e., the local degree of $N_f$ at each simple root of $f$ is at least two. The dynamics of Newton's method applied to polynomials has been extensively investigated; see, for example, \cite{DS2022,HSS2001,Lei1997,Felex1989}. There are root-finding methods such as Chebyshev's method differing from Newton's method in several important and interesting ways.

 For an entire function $f: \mathbb{C} \to \mathbb{C}$, the Chebyshev's method $C_f$ is defined as 
\begin{equation}C_f(z)=z-\left(1+\frac{1}{2}\frac{f(z)f''(z)}{(f'(z))^2}\right)\frac{f(z)}{f'(z)}.
\label{cheby-def}
\end{equation}
 By a Chebyshev (or Newton) map, we mean $C_f$ (or $N_f$ respectively) for some entire function $f$. A root-finding method applied to a function $f$ can have a fixed point that is not a root  of $f$. Such a fixed point is called \textit{extraneous}. Unlike Newton's method, Chebyshev's method can have a finite extraneous fixed point. For a polynomial $p$, an extraneous fixed point of $C_p$ can even be non-repelling (see \cite[Theorem~1.2]{Nayak-Pal2022} and \cite[Section~3.2]{GGM2015}), leading to situations qualitatively different from those occurring for Newton's method. 
In particular, the sequence $C_p^n(z)$ may converge to an extraneous fixed point  for some initial point  $z$, thus defeating the very purpose of a root-finding method, i.e., to  approximate a root of $p$. This undesirable behavior illustrates that the dynamics of Chebyshev maps is considerably richer and more complicated than that of Newton maps. This is probably a reason why Chebyshev's method, despite having third order convergence,  remains comparatively unexplored. From the perspective of complex dynamics, Chebyshev maps are natural candidates for investigation to understand how a higher-order root-finding method may behave as a dynamical system. 
\par

The Chebyshev's method applied to cubic polynomials and unicritical polynomials of arbitrary degree are studied in \cite{GV2020} and  \cite{CCV2020}, respectively.  The degree of $C_p$ is completely determined for every polynomial $p$, and its dynamics is studied for cubic polynomials in \cite{Nayak-Pal2022}.
It follows from Equation (\ref{cheby-def}) that Chebyshev's method applied to a polynomial produces a rational map. However, it is not very difficult to verify that if $f(z)=p(z)e^{q(z)}$ for two polynomials $p$ and $q$, then the 
associated Chebyshev map $C_f$ is also rational. This naturally leads to the question: \textit{What are all entire functions $f$ for which  $C_f$ is rational?} The first result of this article provides a characterization of all such entire functions.
\begin{maintheorem}\label{Characterization}
	Let $f: \mathbb{C} \to \mathbb{C}$ be a non-constant entire function. Then the Chebyshev's method applied to $f$ is rational if and only if $f(z)=p(z)e^{q(z)}$ for two polynomials $p$ and $q$.
\end{maintheorem}
Theorem~\ref{Characterization} is already known for Newton's method, and  its proof is  quite straightforward (\cite[Proposition 2.11]{RC2007}). For Halley's method $H_f (z)= z -\frac{2 f'(z)f(z)}{2 (f'(z))^2 -f(z)f''(z)}$, Theorem~\ref{Characterization} is also true (\cite[Theorem 3.6]{CMHD2022}), where the proof  involves a careful analysis of the  complex differential equation $g'(z) - \frac{2 g(z)}{z- H_f (z)} =-1$ for $g(z)=\frac{f(z)}{f'(z)}$. However, none of these ideas work for the Chebyshev's method. We use some results from Nevanlinna theory for proving Theorem \ref{Characterization}.

\par 
The rest of this paper is devoted to the  study of the Chebyshev's method applied to $p e^q$, where $p$ and $q$  are non-constant polynomials, from a dynamical perspective. Though similar investigations are done on the  Newton's method (see, for example,  \cite{ Haruta1999,Mamayusupov2019,Mayer_Schleicher2006, RC2007}), the authors of the current article are not aware of any such work on the Chebyshev's method. 

\par 

We determine the multiplier of all the fixed points of $C_{pe^q}$ for all polynomials $p$ and $q$. Note that, if $z_0$ is a fixed point of a rational function $R$ (i.e., $R(z_0)=z_0$) then the value $R'(z_0)$ is known as the multiplier of $z_0$  whenever $z_0\in \mathbb{C}$, and if $z_0 =\infty$, then it is taken to be  $G'(0)$ where $G(z)=\frac{1}{R(\frac{1}{z})}$.
The fixed point $z_0$ is called attracting, neutral, or repelling if the modulus of its multiplier is less than, equal to, or bigger than $1$, respectively. An attracting fixed point is called superattracting if the multiplier is $0$. A neutral fixed point is called parabolic or rationally indifferent if its multiplier is a root of unity. A special case arises when the multiplier is $1$, in which $z_0$ is a multiple root of $R(z)-z=0$ with multiplicity $k \geq 2$.  This $k$ is known as the multiplicity of $z_0$ as a fixed point of $R$. It is shown that all the roots of $p$ are attracting fixed points of $C_{pe^q}$, and for a non-constant polynomial $q$, the point $\infty$ is always a parabolic fixed point with multiplicity $\deg(q)+1$ (see Proposition~\ref{Property_C_f}). A formula for the multiplier of an extraneous fixed point is presented in the same Proposition. Note that, for a constant $q$, the point $\infty$ is a repelling fixed point, which is already known (see \cite[Proposition~2.3]{Nayak-Pal2022}). 
\par
 The Julia set of a rational function $R$, denoted by $\mathcal{J}(R)$ is the set of all points,  in a neighborhood of which the family of iterates $\left\{R^n\right\}_{n\geq 0}$ is not normal  in the sense of Montel. The Fatou set of $R$, denoted by $\mathcal{F}(R)$, is the complement of the Julia set in $\widehat{\mathbb{C}}$. The Fatou set is an open set, and each of its maximally connected subsets is known as a Fatou component.  The Julia set is closed. It is connected if and only if all the Fatou components are simply connected. 
 
 There is an important connection between certain types of fixed points and the topology of the Julia set of a rational function. A fixed point is called weakly repelling if it is either repelling or parabolic with multiplier $1$. Shishikura proved that if a rational function $R$ has degree at least $2$ and the Julia set of $R$ is disconnected, then $R$ has at least two weakly repelling fixed points (\cite[Theorem 1]{Shishikura2009}). Since $\infty$ is the only non-attracting fixed point for $N_{pe^q}$ for all polynomials $p$ and $q$, the Julia set of $N_{pe^q}$ is connected. We are concerned with the connectedness of the Julia set, equivalently, the simply connectivity of every Fatou component of rational Chebyshev maps. For this, we consider those with a single attracting fixed point and for which the multiplicity of $\infty$ (as a fixed point) is  any given number $n$.  In order to ensure this, we must take $p$ as linear and $q$ with degree $n$. For $n=1$, it can be seen easily that the Chebyshev's method applied to $p e^q$ is conformally conjugate to that applied to $ze^z$ (mentioned in Remark~\ref{both-linear}(1)). Though $q$ could have been taken to be a  polynomial of degree $n$, we take  $q=(p)^n$, the $n$-th power of $p$, for the sake of a simpler situation. Thus our object of study becomes $C_{ze^{z^n}}$, which we denote  by $C_n$. In fact, it is observed that the Chebyshev's method applied to $p e^{\lambda (p)^n +c} $ for any $\lambda \neq 0, c \in \mathbb{C}$ is conformally conjugate to that applied to  $z e^{z^n}$ (Remark~\ref{SC2}). It is proved that all the extraneous fixed points of  $C_n$  are repelling. We have proved the following.
 \begin{maintheorem}
 	The Julia set of $C_{ze^z}$ is connected.
 	\label{n=1}
 \end{maintheorem}
 For odd $n$ with $n \geq 2$, it is found that $C_n$ has a single real critical point and two real fixed points (both are extraneous and negative). Let $c_r$ and $-e_2$ be the real critical point and the largest real extraneous fixed point of $C_n$ respectively. Then we have shown the following.
	\begin{maintheorem} 
		For odd $n$, if $C_n(c_r)> -e_2 $, then the Julia set of $C_n$ is connected.
		\label{odd-connected} 
	\end{maintheorem}
It is numerically found that the hypothesis of Theorem~\ref{odd-connected} is true for all $n \leq 15$.
For even $n$, the map $C_n$ has a single positive critical point. Denoting this by $c_r$, we have proved the following.
\begin{maintheorem}  
	For even $n$, if  $C_n (x)> -x $ for all $x \in (0,c_r]$, then the Julia set of $C_n$ is connected.
	\label{even-connected}
\end{maintheorem}
It is found that the hypothesis of Theorem \ref{even-connected} is in fact satisfied for some $n$.

\begin{cor}
	If $n$ is even and $n  \leq 16 $, then the Julia set of $C_n$ is connected.
	\label{cor-even}
\end{cor} 
The arguments used in the proof of Corollary~\ref{cor-even} do not work for $n>16$. This shall be taken up in our future work.
\par  For a rational function $R$, a $p$-periodic Fatou component $U$ (i.e., $R^p(U)=U$, but $R^i(U)\neq U$ for any natural number $i<p$) is called a Herman ring (or a Siegel disk) if there exists an analytic homeomorphism $\phi$
such that $R^p:U\to U$ is conformally conjugate to an irrational rotation of some annulus (or of the unit disk) onto
itself. In other words, $\phi(R^p(\phi^{-1}(z)))=e^{i 2 \pi \alpha}z$ for some irrational number $\alpha$ and for all $z$ with $ 1<|z|<r$ (or $|z|<1$). A Fatou component is called a rotation domain if it is either a Herman ring or a Siegel disk. We have shown that, $C_n$ cannot have a Herman ring for any $n$ (Corollary~\ref{No-HR}). Moreover,  it is proved that there is no Siegel disk for any even $n$ (Proposition~\ref{noRD}). 
\par 
 The set of all holomorphic Euclidean isometries (i.e., maps $z \mapsto az +b$ for $|a|=1$ and any complex number $ b$) preserving the Julia set of a rational function $R$ is called its symmetry group, and it is denoted by $\Sigma R$. The relation between the symmetry groups of a polynomial and its Chebyshev's method is investigated in \cite{Sym_dym}. We have shown that $\{z\mapsto \lambda z: \lambda^n=1\}\subseteq \Sigma C_n$ for each $n$ (Lemma \ref{Cheby_sym}).  
 \par The Newton's method applied to $ze^{z^n}$ is reported in \cite{Haruta1999}, and it is shown that the Fatou set of $N_{ze^{z^n}}$ contains the basins of $0$ and $\infty$, which are attracting and parabolic fixed points of $N_{ze^{z^n}}$, respectively. We describe some additional dynamical aspects of $N_{ze^{z^n}}$  including symmetry (Proposition \ref{New_sym}).   
 \par 
The structure of the article is as follows. In Section \ref{characterization-cheby}, we prove Theorem~\ref{Characterization}.
Section \ref{properties-chebyshev} contains the statements and proofs of some useful properties  of rational Chebyshev maps.   Section  \ref{C-n}   is dedicated to the dynamics of   $C_{ze^{z^n}}$, where the proofs of Theorems ~\ref{n=1}, \ref{odd-connected}, and \ref{even-connected} and Corollary~\ref{cor-even} are provided. We study Newton's method applied to $z \mapsto ze^{z^n}$ in Section~\ref{Newton}.
\section{Characterization of rational Chebyshev  maps}
\label{characterization-cheby}
 It is proved in \cite[Proposition 2.11]{RC2007} 
that if $f: \mathbb{C} \to \mathbb{C}$ is an entire function, then its Newton's method $N_f$ is rational if and only if there are polynomials $p$ and $q$ such that $f=pe^q$. The Halley's method for an entire function $f$ is defined as 
 $H_f (z)= z -\frac{2 f'(z)f(z)}{2 (f'(z))^2 -f(z)f''(z)}$. The characterization of rational  Halley maps is the same and is proved in \cite[Theorem 3.6]{CMHD2022}. Taking  $\frac{f}{f'}$ as  $g$, it is seen that \begin{equation}
 g'(z)-\frac{2g (z)}{z-H_f (z)}=-1.
 \label{Halley}
\end{equation} The proof of \cite[Theorem 3.6]{CMHD2022} relies on the fact that Equation~(\ref{Halley}) is a linear and  exact differential equation in $g$ for a given rational $H_f$,  and its solution is $$g(z)= -e^{ \int \frac{2 dz}{ z-H_f (z)}}\left(\int e^{ \int \frac{-2 dz}{ z-H_f (z)}} dz \right).$$ The authors show that $g$ is rational whenever $H_f$ is rational. Then the proof proceeds by analyzing the expression of $\frac{f}{f'}$. We are concerned with a characterization for rational  Chebyshev maps. If one follows the arguments of \cite[Theorem 3.6]{CMHD2022}, then the resulting differential equation becomes $g(z) g'(z)-3 g(z)=2(C_f (z)-z)$, which is non-linear. We use Nevanlinna theory to handle the situation.  Let $f$ be a non-constant meromorphic function in the disc $|z|\leq R$ ($0<R<\infty$). For $0<r<R$, let $m(r,f)=\frac{1}{2\pi}\int_0^{2\pi}\log^+|f(re^{i\theta})|d\theta$, where $\log^+(x)=\max\{\log x,0\}$.  Define $N(r,f)=\int_0^r\frac{n(t,f)-n(0,f)}{t}dt+n(0,f)\log r$, where $n(t,f)$ denotes the number of poles of $f$ in the disc $|z|\leq t$ counting multiplicity. The characteristic function of $f$, denoted by $T(r,f)$, is defined as $T(r,f)=m(r,f)+N(r,f)$.

For two functions $f$ and $g$ from the positive real line into itself, we say $f(x)=o(g(x))$ as $x \to \infty$ if $\lim\limits_{x \to \infty} \frac{f(x)}{g(x)}=0$. Also we say $f(x)=\mathcal{O}(g(x))$ as $x \to \infty$ if  $ \frac{f(x)}{g(x)} \leq K$ for some $K>0$ and for all sufficiently large positive $x$. Note that $f(x)=o(g(x))$ implies $f(x)=\mathcal{O}(g(x))$, but not conversely. We present two lemmas.
\begin{lem}{\rm (\cite[Theorem~1.51]{yy2006})}\label{lem1} 
For $n \geq 2$, suppose that $f_1, f_2, \dots, f_n $ are meromorphic functions and $g_1, g_2, \dots, g_n$ are entire functions satisfying the following conditions.
	\begin{enumerate}
		\item  $\sum\limits_{j=1}^n f_j(z) e^{g_j(z)} \equiv 0$ and   $g_j(z)-g_k(z)$ are non-constant for all  $1 \leq j<k \leq n$. 
\item For $1 \leq j \leq n $ and $ 1 \leq h<k \leq n$, we have 
	$ 
	T\left(r, f_j\right)=o\left(T\left(r, e^{g_h-g_k}\right)\right), ~(r \rightarrow \infty, r \notin E),
	$ 
	where $T(r,f)$ is the Nevanlinna characteristic function of $f$ and $E$ is a set with finite linear measure. 
		\end{enumerate}Then $f_j(z) \equiv 0$ for all $j=1,2, \dots, n$.
\end{lem}

\begin{lem}\label{lem2}
	Let $h$ be a non-constant entire function and $f(z)=e^{h(z)}$.
	Then we have $T(r,h')=o(T(r,f))$ as $r\rightarrow\infty$ and $r$ is outside a set of linear measure not greater than $2$ .
\end{lem}
\begin{proof}
	The function $f$ is a transcendental entire function and it follows from \cite[Theorem~1.5]{yy2006} that
	\begin{equation}
	\lim_{r\rightarrow\infty}\frac{T(r,f)}{\log r}=\infty.
	\label{c1} \end{equation}

	Clearly $T(r,h')=m(r, \frac{f'}{f})$. It follows from \cite[Lemma 1.4']{yy2006}, page 21, that $T(r,h')$ is either $\mathcal{O}(\log r)$ or $\mathcal{O}(\log (rT(r,f)))$ outside of a set $E_0$ of linear measure less than or equal to $2$. If $T(r,h')=\mathcal{O}(\log r)$ then 
	using  Equation~(\ref{c1}), we have $T(r,h')=o(T(r,f))$ as $r\rightarrow\infty$. If $T(r,h')=\mathcal{O}(\log (rT(r,f)))$ then $T(r,h')=K\log (rT(r,f))$ for some $K>0$. Note that $$\lim\limits_{r\rightarrow\infty}\frac{\log (rT(r,f))}{T(r,f)}=\lim\limits_{r\rightarrow\infty}\frac{\log r+\log T(r,f)}{T(r,f)}=\lim\limits_{r\rightarrow\infty}\frac{\log r}{T(r,f)}+\lim\limits_{r\rightarrow\infty}\frac{\log T(r,f)}{T(r,f)}=0$$ outside the set $E_0$. This gives that $T(r,h')=o(T(r,f))$ as $r\rightarrow\infty$ outside of the set $E_0$.
\end{proof}

We now present the proof of Theorem~\ref{Characterization}.
\begin{proof}[ Proof of Theorem~\ref{Characterization}]
If $p$ and $q$ are polynomials and $f=pe^q$, then it is obvious from the Equation (\ref{cheby-def}) that $C_f$ is a rational map.	
	\par
Conversely, let $C_f$ be a rational map. Then $C_f$ has at most finitely many attracting fixed points. Since all the roots of $f$ are attracting fixed points of $C_f$, $f$ has  at most  finitely many roots. This gives that
	\begin{equation}
	f(z)=p(z) e^{q(z)},\label{1}
	\end{equation}
	where $p$ is a polynomial and $q$ is an entire  function.  We assert that $q$ is also a polynomial.
	\par
Since $C_f$ is rational, it has at most finitely many poles. It follows from the definition of the  Chebyshev's method that a pole of $C_f$ is a root of $f^{\prime}$. Let
 \begin{align*}
& I_1=\{z:~f^{\prime}(z)=f(z)=0\},\\
& I_2=\{z:~f^{\prime}(z)=0,~\mbox{but}~f(z)f''(z)\neq0 \},~~ \text{and}\\
& I_3=\{z:~f^{\prime}(z)=f''(z)=0,~\mbox{but}~f(z)\neq0 \}.
\end{align*}
The set $I_1$ is finite since $f$ has at most finitely  many roots. Each point of $I_2$ is clearly a pole of $C_f$. If $z_0\in I_3$, then we can write $f'(z)=(z-z_0)^k \tilde{f_1}(z)$ and $f''(z)=(z-z_0)^{k-1} \tilde{f_2}(z)$ where $k\geq 2$ is a natural number, $\tilde{f_1}$ and $\tilde{f_2}$ are entire functions, both non-vanishing at $z_0$. Therefore, $\frac{f(z) f''(z)}{(f'(z))^2}=\frac{f(z) \tilde{f_2}(z)}{(z-z_0)^{k+1}(\tilde{f_1} (z))^2}$ and $C_f$ has a pole at $z_0$. It is easy to see that each point of $I_3$ and therefore of $I_2 \cup I_3$ is a pole of $C_f$. As the set of poles of $C_f$ is finite, the set $I_2 \cup I_3$ is also finite. Since  $\{z:~f^{\prime}(z)=0\}=I_1\cup I_2\cup I_3$, there are at most finitely many roots of $f^{\prime}$. Since $f^{\prime}$ is also entire,
	\begin{equation}
	f^{\prime}(z)=g(z) e^{h(z)}\label{2}
	\end{equation}	
	for some polynomial $g$ and entire function $h$. From Equations (\ref{1}) and (\ref{2}), we have 
	\begin{equation}\label{q_poly}
	p^{\prime}(z)+p(z) q^{\prime}(z)=g(z)e^{h(z)-q(z)}.
	\end{equation}
	If $h(z)- q(z)$ is a constant function, then we have $p^{\prime}(z)+p(z) q^{\prime}(z)=k g(z)$ for some constant $k$. This gives that $q^{\prime}(z)$ is a polynomial and hence $q(z)$ is also a polynomial. Thus, we are done.
	\par
	Next, we deduce a contradiction when $s(z)=h(z)-q(z)$ is a non-constant entire function. According to the above assumption, we have $f(z)=p(z)e^{q(z)}$ and $$f'(z)=(p'(z)+p(z)q'(z))e^{q(z)}=(g(z)e^{s(z)})e^{q(z)}.$$
	Clearly, $$f''(z)=(g'(z)+g(z)s'(z)+g(z)q'(z))e^{q(z)+s(z)}.$$
	It follows from  Equation~\eqref{q_poly} that
	\begin{align*}
	C_f(z)&=z-\frac{p(z)}{2(g(z))^3}\left(3g^2(z)e^{-s(z)}+(p(z)g'(z)+p(z)g(z)s'(z)-p'(z)g(z))e^{-2s(z)}\right)\\
	&=z-\frac{p(z)}{2(g(z))^3}D_f(z),
	\end{align*}
	where $$
	D_f(z)=3g^2(z)e^{-s(z)}+(p(z)g'(z)-p'(z)g(z)+p(z)g(z)s'(z))e^{-2s(z)}.
$$
We rewrite the above equality as 
	\begin{equation}\label{dif}
	f_1(z)e^{g_1(z)}+f_2(z)e^{g_2(z)}+f_3(z)e^{g_3(z)}=0,
	\end{equation}
	where $f_1(z)=3g^2(z)$, $f_2(z)=p(z)g'(z)-p'(z)g(z)+p(z)g(z)s'(z)$, $f_3(z)=-D_f(z)$, $g_1(z)=-s(z)$, $g_2(z)=-2s(z)$, and $g_3(z)=0$.
	We observe that $T(r,f_1)=\mathcal{O}(\log r)$ as  $f_1(z)$ is a polynomial. Since $C_f(z)$ is rational, we have $D_f(z)$ is rational, and therefore $T(r,f_3)=\mathcal{O}(\log r)$ (see \cite[Theorem 2.2.3]{lai1993}). For $1 \leq i < k \leq 3$, we have $e^{g_i -g_k}$ is either $e^s$,  $e^{-s}$, or $e^{-2s}$. We know that $T(r,e^{-s})=T(r,e^s)+\mathcal{O}(1)$ and $T(r, e^{-2s})=2T(r,e^s)+\mathcal{O}(1)$ (see \cite[Proposition 2.1.11]{lai1993}, page 24). It follows from Equality (\ref{c1}) that, for $j=1$ and $3$, we have  $T(r,f_j)=o(T(r,e^{s}))$. Thus, for $j=1$ and $3$, we have  $T(r,f_j)=o(T(r,e^{g_i -g_k}))$ for $1 \leq i < k \leq 3$.
	For $f_2(z)$, it follows from  \cite[Proposition~2.1.11]{lai1993} that
	\begin{align*}
	T(r,f_2)&=T(r,pg'-p'g+pgs')\\ &\leq T(r,pg')+T(r,-p'g)+T(pgs')+\log 3  ~\\
	&\leq T(r,p)+T(r,g')+T(r,-p')+T(r,g)+T(r,p)+T(r,g)+T(r,s')+\log 3\\
	&\leq2T(r,p)+2T(r,g)+T(r,-p')+T(r,g')+T(r,s')+\log 3 ~\\
	&\leq M\log r+T(r,s')+\log 3 
	\end{align*}
	for some positive integer $M$ and all $r>1$. Therefore, $T(r,f_2)=\mathcal{O}(T(r, s'))$. We have $ T(r,s')=o(T(r,e^s))$ where $r$ is outside of a set $E_0$ of linear measure not greater than $2$ by Lemma \ref{lem2}. This means that $T(r,f_2)=o(T(r, e^s))$ and by~\cite[Proposition~2.1.11]{lai1993}, we have $T(r, f_2)=o(T(r, e^{g_i -g_k}))$ for all $1 \leq i < k \leq 3$ where $r \notin E_0$. Applying Lemma  \ref{lem1}  to the Equation \eqref{dif}, we have $f_i(z)\equiv0$ for $i=1,2,3$ where $r$ is outside of $E_0$. Using the Identity Theorem, we have $f_i(z)\equiv 0$ for $i=1,2,3$.  In particular, $f_1 \equiv 0$. Hence, we get $g(z)\equiv0$, which implies that the function $f$ is a constant. This is a contradiction, and the proof is complete.	
\end{proof}
\section{Properties of rational Chebyshev  maps}
\label{properties-chebyshev}
For two polynomials $p$ and $q$, the Chebyshev's method applied to $f= pe^q$ is \begin{equation}C_{f}=Id- \frac{p(2 (p')^2 +3 p^2 (q')^2+6 p p'q'+pp''+p^2 q'')}{2(p'+pq')^3},
\label{form-chebyshev}\end{equation} where $Id(z)=z$.  It follows from Equation (\ref{form-chebyshev}) that $C_f$ is constant whenever $f$ is a linear polynomial. We consider Chebyshev's method applied to  entire functions that are not linear polynomials throughout this article. It follows from the above that the Chebyshev's method applied to $p(z)e^{q(z)}$ is the same as that applied to $ p(z)e^{q(z)-q(0)}$. We analyze the behavior of $C_f$ using the behavior of Newton's method $N_f (z)=z -\frac{f(z)}{f'(z)}$. Since $N'_f(z)=\frac{f(z)f''(z)}{f'(z)^2}$, we have
\begin{eqnarray}
 C_f(z) 
 &&=z-(1+\frac{1}{2}N'_f(z))(z-N_f(z))\label{conversion}\\
 &&=N_f(z)+\frac{1}{2}N_f'(z)(N_f(z)-z)\label{conversion1}.
\end{eqnarray}
It follows from Equation (\ref{conversion}), that each finite fixed point of $N_f$ is also a fixed point of $C_f$. Note that $\infty$ is a fixed point of $N_f$. The following proposition discusses the fixed points of $C_f$.

\begin{prop} (Fixed points) \label{Property_C_f}
	Let $C_f$ be the Chebyshev's method applied to $f(z)=p(z)e^{q(z)}$, where $p$ and $q$ are non-constant polynomials. Then,
	\begin{enumerate}
		\item If $\alpha$ is a root of $p$ with multiplicity $k$, then its multiplier is $ \frac{2 k^2 -3k +1}{2k^2}$. In particular, a simple root of $p$ is a superattracting fixed point of $C_f$.
        \item The point at infinity is a multiple fixed point of $C_f$ with multiplicity $\deg(q)+1$ and therefore, it is parabolic.
        \item The finite extraneous fixed points of $C_f$ are the solutions of  $1+\frac{1}{2}N_f'(z)=0$. Moreover, if $e_0$ is an extraneous fixed point of $C_f$, then its multiplier is given by $ 1-\frac{1}{2}N_f''(e_0)\frac{f(e_0)}{f'(e_0)}$.
	\end{enumerate}
\end{prop}
\begin{proof}
\begin{enumerate}
\item For $z$ in a neighbourhood of $\alpha$, $f(z)=(z-\alpha)^{k}g(z)$, where $g(z)$ is analytic at $\alpha$ with $g(\alpha)\neq 0$. Consequently, in a punctured disc around $\alpha$, we have 
\begin{align}\label{newton}
\nonumber  N_f(z)& 
=z-\frac{(z-\alpha)^{k}g(z)}{k(z-\alpha)^{k-1}g(z)+(z-\alpha)^{k}g'(z)} \\
&\nonumber  =z-\frac{(z-\alpha)g(z)}{k g(z)+(z-\alpha)g'(z)}=z-\frac{(z-\alpha)}{k}\frac{1}{1+(z-\alpha)\frac{g'(z)}{k  g(z)}}\\
&
=z-\frac{(z-\alpha)}{k}(1+\mathcal{O}(z-\alpha)),
\end{align} 
where $\mathcal{O}((z-\alpha)^k)$ denotes a power series in $(z-\alpha )$ in which the smallest power of $(z-\alpha )$ is  $k\geq 1$. Note that this is the same as defined in the proof of Theorem~\ref{Characterization}. Therefore, 
\begin{equation}\label{derinewton}
N'_f(z)= (1-\frac{1}{k})+\mathcal{O}(z-\alpha).
\end{equation}
Now from Equations (\ref{conversion}), (\ref{newton}), and (\ref{derinewton}),
\begin{align*}\label{chebyfixed}
C_f(z)&=
z-\left(1+\frac{1}{2}(1-\frac{1}{k })+\mathcal{O}(z-\alpha )\right)\left(\frac{1}{k }(z-\alpha )+\mathcal{O}(z-\alpha )^2\right)\\
&=z-\frac{1}{2k }(3-\frac{1}{k })(z-\alpha )+\mathcal{O}(z-\alpha)^2 
\end{align*}
Therefore, $C_f(\alpha)=\alpha$ and $C'_f(\alpha )= \frac{2k^2 -3k +1}{2k^2 }$, which is less than $1$ for all $k \geq 1$. 

  In particular, if $\alpha$ is a simple root of $p$, then it is a superattracting fixed point of $C_f$.
\item 
Note that $\infty$ is a fixed point of $N_f$ with multiplicity $\deg(q)+1$ (see \cite[Proposition 2.11]{RC2007}) and $ g(z) = \frac{1}{N_f(\frac{1}{z})}  =z+\frac{z^{n+1}}{n}+\mathcal{O}(z^{n+2})$  near the origin where $n =\deg(q)$. Then  $N_f (z) =\frac{1}{g(\frac{1}{z})}$ for all $z$ in a neighborhood of $\infty$. Thus $N_f (z)= \frac{1}{\frac{1}{z} +\frac{1}{nz^{n+1}} +\mathcal{O}(\frac{1}{z^{n+2}})}= z ( 1+\frac{1}{nz^{n}}+\mathcal{O}(\frac{1}{z^{n+1}}))^{-1}$, and this is nothing but $N_f(z)=z-\frac{1}{nz^{n-1}}+\mathcal{O}(\frac{1}{z^n}).$

 Therefore, $N_f'(z)=1+\frac{n-1}{nz^n}+\mathcal{O}\left(\frac{1}{z^{n+1}}\right)$ for all $z$ in a neighborhood of $\infty$ and, by Equation (\ref{conversion}),  we have 
\begin{align}
\nonumber C_f(z) 
&=z-\left(1+\frac{1}{2} +\frac{n-1}{2n}\frac{1}{z^n}+\mathcal{O}\left(\frac{1}{z^{n+1}}\right)\right)\left(\frac{1}{nz^{n-1}}+\mathcal{O}\left(\frac{1}{z^{n}}\right)\right)\\
& =z-\frac{3}{2n}\frac{1}{z^{n-1}}+\mathcal{O}\left(\frac{1}{z^{n}}\right).
\end{align}
It can be seen that $\frac{1}{C_f (\frac{1}{z})}= z+\frac{3}{2n}z^{n+1}+\mathcal{O}(z^{n+2})$, which means that $\infty$ is a multiple fixed point of $C_f$ with multiplicity $ n+1 $. Hence, it is parabolic.

\item It is clear from Equation (\ref{conversion1}) that each finite extraneous fixed point of $C_f$ is a solution of $1+\frac{1}{2}N_f'(z)=0$. It can be deduced that $C_f'(z)=1-\frac{1}{2}N_f''(z)\frac{f(z)}{f'(z)}-(1+\frac{1}{2}N_f'(z))(1-N_f'(z)).$ For the extraneous fixed point $e_0$, we get $1+\frac{1}{2}N_f'(e_0)=0$. Therefore, the multiplier of $e_0$,  $ C_f'(e_0)=1-\frac{1}{2}N_f''(e_0)\frac{f(e_0)}{f'(e_0)}.$  
\end{enumerate}
\end{proof}
\begin{rem}
	It follows from the above proof that Proposition~\ref{Property_C_f}(1) is true for every entire function. 
\end{rem}
\begin{rem}\label{Infinity_J_pt}
    By \cite[Proposition 2.3(2)]{Nayak-Pal2022} and Proposition~\ref{Property_C_f}(2), if $f$ is not a linear polynomial and is of the form $f=pe^q$, where $p$ and $q$ are not simultaneously constant, then the point at $\infty$ is either a repelling fixed point (when $q$ is constant) or a parabolic fixed point. Consequently, $\infty$ always belongs to the Julia set of $C_{pe^q}$.
		
	Moreover, when $q$ is non-constant, the degree of $q$ determines the precise nature of $\infty$ as a fixed point of $C_f$. In this case, $\infty$ is a parabolic fixed point with $\deg(q)$ attracting petals. This follows from  the Petal Theorem (see~\cite[Theorem~6.5.4]{Beardon_book}).
\end{rem}
Distinct polynomials may lead to the same function up to conjugacy when the Chebyshev's method is applied to them. 
This is a consequence of the so-called Scaling property of the Chebyshev's method applied to polynomials (\cite[Theorem 2.2]{Nayak-Pal2022}). It follows from the same proof that the Scaling property holds for the Chebyshev's method applied to any entire  function. We present this here.

\begin{prop}(Scaling property)\label{scaling}
Suppose that $a,b,\lambda \in \mathbb{C},a \neq 0, \lambda \neq 0$ and $T(z)=az+b$. If $f$ is an entire function  and $g(z)=\lambda f(T(z))$, then $(T\circ C_g \circ T^{-1})(z)=C_f(z)$ for all $z$. In particular, $C_g (z)=C_{\lambda f}(z)$ for all $\lambda$.
\end{prop}

This article is concerned with the Chebyshev's method applied to $pe^q$, where $p $ is a linear polynomial and $q =p^n$. It is sufficient to consider $p(z)=z$. The next remark implies that something slightly more than this,  is actually true.
\begin{rem}\label{SC2} For  $ a, b, c, \lambda \in \mathbb{C}, a \neq 0, \lambda \neq 0$, let $p(z)=a z+b$ and $f(z)=p(z)e^{\lambda(p(z))^n +c}$.
	It is clear from Equation (\ref{form-chebyshev}) that $c=0$ can be taken without loss of generality. Consider $T(z)=\frac{1}{a}(\alpha z-b)$ where $\alpha^n=\frac{1}{\lambda}$. Then $g(z) =\frac{1}{\alpha} f(T(z))=\frac{1}{\alpha} f(\frac{1}{a}(\alpha z-b))=z e^{z^n}$. Now, $C_g $ is conjugate to $C_f$ by the Scaling property (Proposition \ref{scaling}).  
\end{rem}
Here is also a useful remark.
\begin{rem}
	\begin{enumerate}
\item 
If both $p$ and $q$ are linear polynomials, i.e., $p(z)=az +b$ and $q(z)=cz+d$ for some $a,b,c,d \in \mathbb{C}$ and $ a,c \neq 0$, then considering $T(z)= \frac{z }{c}-\frac{b}{a}$ and $\lambda =\frac{c}{a}$, it is seen, by the Scaling property, that $C_{pe^q}$ is affine conjugate to the Chebyshev's method applied to $ze^{z-\frac{bc}{a}+d} $. Now the Chebyshev's method applied to this function is the same as that applied to $z e^z$.
	\item If $f=pe^q$ where $p$ is constant and $q$ is non-linear, then $C_{f}(z)=z-\frac{3}{2 q'}-\frac{q''}{2(q')^3}$. If $q(z)=z$, then $C_f (z)=z-\frac{3}{2}$. For $q(z)=z^n, n \geq 2$ we have $C_f (z)=\frac{2n^2 z^{2n}-3n z^n -n +1}{2n^2 z^{2n -1}}$. It can be easily seen that each extraneous fixed point of $C_f$ is a solution of $z^n =\frac{1-n}{3n}$, and the multiplier of each such fixed point is $1+\frac{9n}{2(n-1)}$. Therefore, all the extraneous fixed points are repelling. 
		\end{enumerate}
	\label{both-linear}
\end{rem}
It is proved in \cite[Lemma 2.10]{RC2007}, that if $g$ is an entire function and the immediate basin of an attracting fixed point of its Newton's method $N_g$ corresponding to a root of $g$ is equal to the whole complex plane, then $g(z)=a (z-\xi)^d$ for some $d >0 $ and complex numbers $a$ and $\xi$. Essentially the same result is proved for the  Halley’s method (\cite[Proposition 3.4]{CMHD2022}). For the Chebyshev's method, we establish the following, which is slightly more general in the sense that we do not assume the existence of an attracting fixed point.
\begin{prop}\label{Entire_C}
Let $g$ be a non-constant entire function. If the Fatou set of $C_g$ is precisely the whole complex plane, then $C_g$ is affine conjugate to $C_f$, where $f(z)=z^d$ for some $d \geq 2$ or $f(z)=e^z$.
\end{prop}
\begin{proof}
No Fatou component of any meromorphic function can contain two periodic points. If the Fatou set of such a function is the whole complex plane, then it contains at most one finite periodic point, and that must be a fixed point. Since each transcendental meromorphic function has infinitely many periodic points with minimal period $m$ for $m\geq 2$ (by \cite[Theorem 2]{Ber1993}), the function $C_g$ cannot be transcendental and therefore, it is a non-constant rational function. Hence, from Theorem~\ref{Characterization}, we find that $g$ is of the form $g=pe^q$, where $p$ and $q$ are polynomials. It follows from Remark~\ref{Infinity_J_pt} that $\infty$ is either a repelling or a parabolic fixed point of $C_g$, and therefore belongs to $\mathcal{J}(C_g)$. Thus, the local degree of $C_g$ at $\infty$ is one. Since there is no finite pre-image of $\infty$, $C_g(z)=Az+B$ for some $A, B \in \mathbb{C}, A \neq 0$.

\par Note that if $f(z)=z^d, d \geq 2$ or $e^z$, then $C_{f} (z)=\frac{2 d^2 -3d +1}{2 d^2}z$ or $z-\frac{3}{2}$, respectively. The proof will be complete by showing that $C_g (z)=Az +B$ is affine conjugate to one of these functions. If $C_g$ has a finite fixed point, i.e., $A\neq 1$, then that finite fixed point cannot be extraneous. If it is so, then $g$ cannot have a root, and consequently $g=Me^q$, where $M$ is a non-zero constant and $q$ is a non-constant polynomial. As $\infty$ is a fixed point of $C_g$ and  $deg(C_g)=1$, the multiplicity of $\infty$ as a fixed point of $C_g$ is at most $2$. Since $q$ is a non-constant polynomial, it follows from Proposition~\ref{Property_C_f}(2) that $q$ is a linear polynomial. But in that case, $C_g$ is a non-identity translation and therefore has no finite fixed point, leading to a contradiction. Hence, the finite fixed point of $C_g$ must be a root of $g$, and therefore an attracting fixed point of $C_g$. It follows from Proposition~\ref{Property_C_f} that its multiplier is of the form $\frac{2 d^2 -3d +1}{2 d^2}$ for some $d\geq 2$. Thus, $C_g$ is conjugate to $\frac{2 d^2 -3d +1}{2 d^2}z$. 
\par 
If $C_g$ has no finite fixed point, i.e., $A=1$, then $\infty$ is the only fixed point and, therefore, $C_g$ is conjugate to $z-\frac{3}{2}$.
 \end{proof}

\section{The Chebyshev's method applied to $ze^{z^n}$}
\label{C-n}
Recall that the Chebyshev's method applied to $f(z)=ze^{z^n}$ is denoted by $C_n$. We have 
 \begin{eqnarray}\label{formula_cheby}
C_n(z)=\frac{nz^{n+1}(2n^2z^{2n}+3nz^n-n+1)}{2(nz^n+1)^3}
\end{eqnarray} and 
\begin{equation}\label{dericheby}
C'_n (z)=\frac{nz^n\left(2n^3z^{3n}+n^2(3n+5)z^{2n}+n(2n^2+3n+4)z^n-(n^2-1)\right)}{2(nz^n+1)^4}.
\end{equation}
 
 \subsection{Properties of $C_n$}
  By saying that the dynamics of a rational function $R$ is preserved by a map $\phi$, we mean the Julia set (hence the Fatou set) of $R$ is preserved by $\phi$. Now, we provide some maps preserving the dynamics of $C_n$.
\begin{lem}\label{Cheby_sym}
 The function $C_n$ is conjugate to itself via the maps $z\mapsto \lambda z$, where $\lambda^n=1$, and $z\mapsto \bar{z}$. Consequently, the Julia set of $C_n$ is preserved under the rotations  $\{ z \mapsto \lambda z: \lambda^n =1\}$ and under the map $ z \mapsto \bar{z}$. 
\end{lem} 
\begin{proof}
It follows from Equation (\ref{formula_cheby}) that $$C_n(\lambda z)=\frac{n \lambda z^{n+1}\left(2 n^2 z^{2 n}+3 n z^n-(n-1)\right)}{2\left(n z^n+1\right)^3} =\lambda C_n(z).$$ 
Thus $C_n$ is conjugate to itself via the map $z \mapsto \lambda z$ for $\lambda^n =1$. Similarly, since all the coefficients of the numerator and denominator polynomials of $C_n$ are real, we have  $\overline{C_n (\bar{z})}=C_n (z)$ for all $z$. 
\par  The Julia set of $C_n$ is preserved under a map $\phi$ if $\phi \circ C_n \circ \phi^{-1}=C_n$ (see \cite[Theorem 3.1.4]{Beardon_book}). Thus, all the rotations $\{ z \mapsto \lambda z: \lambda^n =1\}$ and $ z \mapsto \bar{z}$ preserve the Julia set of $C_n$.
\end{proof} 
\begin{rem}
 If $n$ is even, then $C_n$ is an odd function, i.e., $-C_n (-z) =C_n (z)$. It follows from Lemma~\ref{Cheby_sym} that  $\mathcal{J}(C_n)$ is preserved under the reflection about the imaginary axis $z \mapsto -\bar{z}$ as well.
\end{rem}
 It follows from Proposition \ref{Property_C_f} that the origin is a superattracting fixed point of $C_n$ and $\infty$ is a parabolic fixed point with multiplicity $n+1$.  Also, from Equation (\ref{formula_cheby}) we get that the local degree of $C_n$ at the origin is $n+1$. Now we study all other fixed points. 
\begin{prop}\label{Extrn}
 	For $ n \geq 1$, the total number of extraneous fixed points of $C_n$ is $2n$, and all are repelling. Moreover, the set of extraneous fixed points is preserved under the maps $z \mapsto \lambda z$ where $\lambda^n=1$.
\end{prop}
\begin{proof}
It follows from Proposition~\ref{Property_C_f}(3) that the finite extraneous fixed points of $C_n$ are the solutions of $1+\frac{nz^n(nz^n+n+1)}{2(nz^n+1)^2}=0$, i.e., the solutions of 
	\begin{equation}\label{ex1}
		3n^2z^{2n}+n(n+5)z^n+2=0.
	\end{equation}
	Considering $w=z^n$, Equation (\ref{ex1}) becomes $3n^2w^2+n(n+5)w+2=0$, whose solutions are $$w=\frac{-(n+5)\pm \sqrt{n^2+10n+1}}{6n}.$$ 
	Since $\sqrt{n^2+10n+1}<n+5$, we have $w<0$ for all $n\geq 1$. Each extraneous fixed point of $C_n$ is an $n^{th}$ root of $w$ and satisfies one of the following equations: 
	\begin{equation}\label{ext_f1}
		z^n=\frac{-(n+5)+\sqrt{n^2+10n+1}}{6n}.
	\end{equation}
	
	\begin{equation}\label{ext_f2}
		z^n=\frac{-(n+5)- \sqrt{n^2+10n+1}}{6n}.
	\end{equation}
	
 Let $e_0$ be an extraneous fixed point of $C_n$. Then $N_f(e_0)=\frac{ne_0^{n+1}}{1+ne_0^n}$, $N'_f(e_0)=\frac{n[ne_0^{2n}+(n+1)e_0^n]}{(1+ne_0^n)^2}$, and $$N''_f(e_0)=\frac{n^2e_0^{n-1}[-n(n-1)e_0^n+n+1]}{(ne_0^n+1)^3}.$$
Therefore,  the multiplier  $\lambda$  of $e_0$ is $$\lambda=1-\frac{1}{2}N_f''(e_0)\frac{f(e_0)}{f'(e_0)}=1+\frac{n^2e_0^n[n(n-1)e_0^n-(n+1)]}{2(ne_0^n+1)^4}.$$
Since $e_0$ satisfies either Equations (\ref{ext_f1}) or (\ref{ext_f2}), and $\sqrt{n^2+10n+1}<n+5$, we have $ne_0 ^n<0$. Therefore $\lambda >1$. Hence, all the extraneous fixed points of $C_n$ are repelling. The rest of the theorem follows directly from Equations (\ref{ext_f1}) and (\ref{ext_f2}).
\end{proof}
The right-hand sides of Equations (\ref{ext_f1}) and (\ref{ext_f2}) are negative for all $n$, and the existence of real extraneous fixed points is now immediate. 
\begin{rem}
If $n$ is odd, then there are two real extraneous fixed points, and both are negative. If $n$ is even, there is no real extraneous fixed point.
\label{real-extra}
\end{rem}
The case $n=1$ is slightly different from the point of view of critical points. For $f(z)=z e^z$, there are three critical points, namely $\frac{-4\pm i\sqrt{2}}{2}$ and $ 0$ of  $C_f$, and the first two  are simple whereas   $0$ is  with multiplicity two. 
\par 
 Let  $n\geq 2$.
As the degree of $C_n$ is $3n+1$, the map $C_n$ has $6n$ critical points counting with multiplicity. Among them, the point $0$ is with multiplicity $n$, and it lies in the Fatou set of $C_n$. There are $n$ poles of $C_n$, and all those are in $\mathcal{J}(C_n)$, each of which is a critical point with multiplicity two. The remaining $3n$ critical points are referred to as \textit{free critical points}. Their forward orbits need to be investigated for determining the dynamics of $C_n$ completely.
\begin{lem}
If $n \geq 2$ and $F(w)=2n^3w^{3}+n^2(3n+5)w^{2}+n(2n^2+3n+4)w-(n^2-1)$, then the set of all free critical points of $C_n$ is partitioned into three sets $\{z: z^n=r\}$, $\{z: z^n=c\}$, and $\{z: z^n=\bar{c}\}$ where $r$ and $c$ are the real (positive) and non-real roots of $F(w)=0$, respectively. Furthermore, $r \in (0,1)$.
\label{criticalpoints}
\end{lem}
\begin{proof}
It follows from Equation (\ref{dericheby}) that the free critical points of $C_n$ are the solutions of
\begin{equation}\label{cri}
2n^3z^{3n}+n^2(3n+5)z^{2n}+n(2n^2+3n+4)z^n-(n^2-1)=0.
\end{equation}
Substituting $w=z^n$ in Equation (\ref{cri}), we have $F(w)=0$, where
\begin{equation}\label{wcri}
F(w)=2n^3w^{3}+n^2(3n+5)w^{2}+n(2n^2+3n+4)w-(n^2-1).
\end{equation}
Since $F(0)=-(n^2-1)<0$ and $F(1)=7n^3+7n^2+4n+1>0$,  $F(w)=0$  has a real root between $0$ and $1$. Let this be denoted by $r$. Since $F'(w)=n(6n^2w^2+2n(3n+5)w+2n^2+3n+4)>0$ for all $w>0$, $r$ is the only positive root of $F$. We are to show that it is the only real root of $F$.
\par
For $n >4$, all the roots of $F'(w)=0$ are non-real  as the discriminant $$4n^2(3n+5)^2-24n^2(2n^2+3n+4)=-4n^2[3n(n-4)-1]<0.$$ Since $F'(0)>0$, we have $F'(w)>0$ for all real $w$, and therefore $F$ is strictly increasing. Thus, $r$ is the only real root of $F$ for all $n>4$.
\par  
 For $n=2,3$ or $4$, we have $-4n^2[3n(n-4)-1]>0$, and therefore $F'(w)$ has two distinct real roots, say $x_1, x_2$ with $x_1 < x_2$. Further, $ x_1, x_2 <0$ and $F'(w)>0$ for all $w<x_1$, $F'(w)<0$ for all $x_1 < w<x_2$ and $F'(w)>0$ for all $ w>x_2$. This implies that $F$ is increasing in $(-\infty, x_1)$, decreasing in $(x_1, x_2)$, and increasing thereafter. It can be seen by simple calculation that $F(x_1)<0$. This, along with $F(0)<0$, shows that $F$ has no negative root. Thus $r$ is the only real root of $F$ for $n = 2,3,4$.

\par The other two roots of $F$ are non-real and are complex conjugates of each other. Let them be denoted by $~c  \text{ and } \bar{c}$. Thus, the critical points of $C_n$ are the solutions of 
\begin{equation}\label{set_cr_pt}
    z^n=r,~ z^n=c, \text{ and } z^n=\bar{c}.
\end{equation}
\par 
The set of all free critical points of $C_n$ is the union of $\{z: z^n=r\}$, $\{z: z^n=c\}$, and $\{z: z^n=\bar{c}\}$.  
\end{proof}
\begin{rem}\label{cr_rot}
\begin{enumerate}
\item The set of all critical points of $C_n$ is invariant under the rotations $z \mapsto \lambda z$ with $\lambda^n =1$. 
\item It follows from Lemma~\ref{Cheby_sym} that  $C_n$ is self-conjugate via the maps $z\mapsto \lambda z$, where $\lambda^n=1$, and $z\mapsto \bar{z}$. Now, Lemma~\ref{criticalpoints} gives that there are essentially two free critical points, namely an $n$-th root of $c$ and the positive $n$-th root of $r$.
\end{enumerate}
\end{rem}
Now we look at the zeros of $C_n$. For $n=1$, it is seen from Equation (\ref{formula_cheby})  that the origin is a root with multiplicity three and the other root is $-\frac{3}{2}$. 
\begin{lem}
	For $n >1$, there are three distinct real roots of $C_n$, including one at the origin.
	\label{roots}
\end{lem} 
\begin{proof}
For $n >1$, the origin is a zero of $C_n$ with multiplicity $n+1$. The other zeros are the solutions of $ 2 n^2 z^{2 n}+3 n z^n-(n-1)=0$.
These are the solutions of
$$
z^n=\frac{-3+\sqrt{8 n+1}}{4 n} \text { or } z^n=\frac{-3-\sqrt{8 n+1}}{4 n} \text {. }
$$
If $n$ is even, then there are two real solutions of $z^n=\frac{-3+\sqrt{8 n+1}}{4 n}$ and no real solution of $z^n=\frac{-3-\sqrt{8n+1}}{4 n}$. If $n$ is odd, then there is one real solution of $z^n=\frac{-3+\sqrt{8n+1}}{4 n}$ and also one real solution of $z^n=\frac{-3-\sqrt{8n+1}}{4 n}$. In each of these cases, $C_n$ has exactly two non-zero real zeroes. For even $n$, these are of the same modulus.
\end{proof}
 We denote the positive zero of $C_n$ by $z_0$ when $n$ is even. For odd $n$, let the positive and the negative zeros of $C_n$ be denoted by $z_{+}$ and $z_{-}$, respectively.

\subsection{The dynamics of $C_n$}
The basin of an attracting fixed point $z_0$ of a rational function $R$ is defined as the set of all points $z \in \widehat{\mathbb{C}}$ for which $\lim\limits_{m \to \infty} R^m (z)=z_0$. This set is an open subset of the Fatou set $\mathcal{F}(R)$ of $R$, but it is not necessarily connected. Its connected component containing $z_0$ is called the immediate basin of $z_0$. The basin of a parabolic fixed point $z_0$ of $R$ is defined slightly differently. It is the set of all points $z \in \widehat{\mathbb{C}}$ for which $R^m (z) \neq z_0$ for any $m$ and $\lim\limits_{m \to \infty} R^m (z)=z_0$. The immediate basin of $z_0$ is the union of all the connected components of its basin containing $z_0$ on their common boundary.  We denote the basin of attraction of an attracting or parabolic fixed point $z_0$ of a rational function $R$ by $\mathcal{B}_{z_0}$ (or $\mathcal{B}_{R,z_0}$, whenever the function $R$ is required to be mentioned) and the immediate basin of $z_0$ by $\mathcal{A}_{z_0}$ (or $\mathcal{A}_{R,z_0}$). 
\par
Recall that $\infty$ is a rationally indifferent fixed point of $C_n $ with multiplicity $(n + 1)$ and let $\mathcal{A}_{\infty}$ denote the immediate basin of $\infty$.
\begin{lem}\label{pole_bdry}
	 The boundary of the immediate basin of $\infty$ contains all the poles of $C_n$.
\end{lem}
\begin{proof}
	The map $C_n$ has $n$ distinct poles, say $\omega_1, \omega_2, \ldots, \omega_n$, which are preserved under the rotations of order $n$ about the origin. The local degree of $C_n$ at each pole is $3$.
\par  Consider a sufficiently small neighbourhood of $\infty$, $B_s=\{z: \rho(z, \infty)<s\}$ where $s>0$ and $\rho$ is the spherical metric, such that $B_s$ doesn't contain any finite critical value. As $B_s$ is chosen to be sufficiently small, the set $C_n^{-1}(B_s)$ has $n+1$ distinct components, one of which, say, $N_\infty$ contains $\infty$. Every other component  contains a single pole. Denote the component containing $\omega_i$ by $N_{i},$ for $ i=1,2, \ldots, n$ and observe that the degree of $C_n$ on each $N_{i}$ is three. Further, $C_n$ is one-one in $N_\infty$.
\par 
Let $\xi \in B_s \cap \mathcal{A}_{\infty}$. As the local degree of $C_n$ in $\mathcal{A}_{\infty}$ is at least two (since $\mathcal{A}_{\infty}$ contains at least one critical point of $C_n$), $\xi$ has at least two preimages in $\mathcal{A}_{\infty}$; one is in $N_\infty$ and the others are in $N_{i} \cap \mathcal{A}_{\infty}$ for some $i \in\{1,2, \ldots, n\}$. This is true for every $s'<s$ and for each non-critical value $\xi $ belonging to $ B_s \cap \mathcal{A}_{\infty}$. Considering a sequence $ s_n    \rightarrow 0$ and $\xi_n \in B_{s_n} \cap \mathcal{A}_{\infty}$ such that each $\xi_k$ is distinct and $\xi_k \rightarrow \infty$, we get a sequence $z_k$ in $(  \bigcup\limits_{i=1}^{n} N_{i}) \cap \mathcal{A}_{\infty}$ such that $C_n\left(z_k\right)=\xi_k$.
Since $z_k \neq z_{k'}$ for all $k \neq  k'$, there is a subsequence $\left\{z_{k_n}\right\}$ and $j^* \in\left\{1,2, \ldots, n\right\}$ such that $z_{k_m} \in N_{j *} \cap \mathcal{A}_{\infty}$ for all $m$. This subsequence has a limit point, and that cannot be anything but a pole of $C_n$. Since each $z_{k_n}  \in \mathcal{A}_{\infty}$, this pole is on the boundary $\partial \mathcal{A}_{\infty}$ of $\mathcal{A}_{\infty}$. Therefore, $\partial \mathcal{A}_{\infty}$ contains one pole. By Lemma~\ref{Cheby_sym}, the Fatou set $\mathcal{F}(C_n)$ is preserved under rotations of order $n$. Since $\infty$ is fixed under these rotations, each component of $\mathcal{A}_{\infty}$ is preserved under these rotations. As the poles of $C_n$ are also preserved under these rotations, we have that  $\partial \mathcal{A}_{\infty}$ contains all the poles of $C_n$.
\end{proof}	
Using Lemma \ref{pole_bdry}, we obtain some situations under which the Julia set is connected.

\begin{prop}\label{Connected_J_set}
If $U$ is a Fatou component such that $C_n ^{k}(U) \cap \mathcal{A}_\infty =\emptyset$ for each $k$, then $U$ is simply connected. In particular, the immediate basin $\mathcal{A}_{0}$ of $0$ is simply connected. Moreover, if either $\mathcal{A}_{0}$ is unbounded or each component of $\mathcal{A}_\infty$ is simply connected, then the Julia set of $ C_n $ is connected.
\end{prop}
\begin{proof}
	If $U$ is not simply connected, then there exists a Jordan curve $\gamma$ in $U$ surrounding some bounded component of the Julia set, i.e., some bounded component of the Julia set is contained in a bounded component of $\widehat{\mathbb{C}} \setminus \gamma$. As the closure of the backward orbit of a pole is dense in the Julia set, there  exists an $m$ such that $C_n^m(\gamma)$ is a closed curve surrounding a pole. By assumption, $C_n^m(U) \cap \mathcal{A}_{\infty} = \emptyset$. Therefore,   $\mathcal{A}_\infty$ is surrounded by $C_n^m(\gamma)$  by virtue of Lemma~\ref{pole_bdry}. But this is a contradiction to the fact that $\mathcal{A}_\infty$ is unbounded. Therefore, $U$ is simply connected. Since $\mathcal{A}_0$ is invariant, $C_n ^{k} (\mathcal{A}_0)$ is disjoint from $\mathcal{A}_\infty$ for each $k$. Therefore, $\mathcal{A}_0$  is simply connected.
 \par 
     To prove the second part of the theorem, note that an unbounded invariant attracting domain contains at least one pole on its boundary (see \cite[Lemma 4.3]{Nayak-Pal2022}). If $\mathcal{A}_0$ is unbounded, then the boundary of $\mathcal{A}_0$ contains at least one pole, hence all the poles of $C_n$. This is because each rotation of order $n$ preserves the set of all poles as well as $\mathcal{A}_0$.  
     The boundary of $\mathcal{A}_0$ is an unbounded Julia component containing all the poles.
     If $\mathcal{A}_{\infty}$ is simply connected, then its boundary is also an unbounded Julia component containing all the poles by Lemma~\ref{pole_bdry}. Now, if there is a multiply connected Fatou component, then by considering a non-contractible Jordan curve in it and applying the same argument as in the first paragraph of this lemma, we get a contradiction. 
     
     This proves that each Fatou component is simply connected. In other words, $\mathcal{J}(C_n)$ is connected.   
\end{proof}
A Herman ring is doubly connected by definition. Since its iterated image cannot intersect the parabolic basin of $\infty$, we have the following consequence. 
\begin{cor}
There is no Herman ring for $C_n$ for any $n$.
\label{No-HR} 	
\end{cor}
Now, we prove Theorem~\ref{n=1}.
\begin{proof}[Proof of Theorem~\ref{n=1}]
Note that $$C_1(z)=\frac{z^3(2z+3)}{2(z+1)^3} ~~\text{and}~~ C_1'(z)=\frac{z^2(2z^2+8z+9)}{2(z+1)^4}.$$ Apart from the origin, $C_1$ has two free critical points $\frac{-4\pm i\sqrt{2}}{2}$. As $\mathcal{J}(C_1)$ is symmetric about the real axis and $\mathcal{A}_\infty$ contains at least one critical point, these two free critical points are in $\mathcal{A}_\infty$. Therefore, the immediate basin of $0$, $\mathcal{A}_0$ contains no critical point other than the origin. Note that $(0,\infty)$ is invariant under $C_1$,  and $C_1'(x)>0$ for every $x>0$. Again, as $C_1(x)-x=-\frac{x(3x^2+6x+2)}{2(x+1)^3}$, we have  $C_1(x)<x$  for all $x>0$. This implies that $\lim\limits_{k\to \infty}C_1^k(x)=0$ for all $x >0$ (Figure~\ref{plot_C_1}). Therefore, we get that $\mathcal{A}_0$ is unbounded. Thus, by Proposition \ref{Connected_J_set} we conclude that $\mathcal{J}(C_1)$ is connected.
\end{proof}
 Figure~\ref{Fatouset-C_1} describes the dynamical plane of $C_1$, where the attracting basin of $0$ is shown yellow, and the parabolic basin of $\infty$ in blue.
\begin{figure}[h!]
	\begin{subfigure}{0.5475\textwidth}
		\centering
		\includegraphics[width=1.0\linewidth]{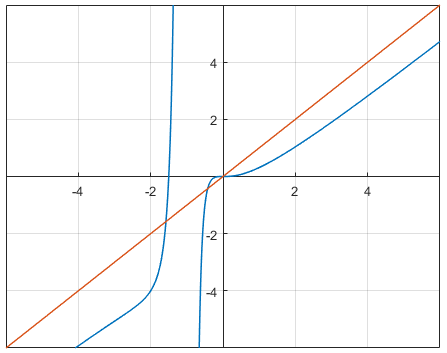}
		\caption{The graph of $C_1$ on the real axis}
		\label{plot_C_1}
	\end{subfigure}
	\begin{subfigure}{.4216\textwidth}
		\centering
		\includegraphics[width=1\linewidth]{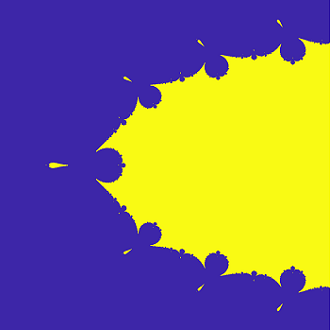}
		\caption{The Fatou and Julia sets of $C_1$}
		\label{Fatouset-C_1}
	\end{subfigure}
	\caption{The Chebyshev's method applied to $ze^z$}
\end{figure}

From now on, we consider $n\geq 2$. Let the positive real solution of the equation $z^n=r$ be $c_r$, i.e., $c_r$ is a critical point of $C_n$ lying on the positive real axis (see Lemma~\ref{criticalpoints}).
Now for real $x$, we can write $C_n'$ as
\begin{equation}\label{crit_C_f}
C_n'(x)=\frac{x^n\left(x^n-c_r^n\right)\left(x^n-c\right)\left(x^n-\bar{c}\right)}{(x^n+\frac{1}{n})^4} =\frac{x^n (x^n -c_r ^{n})((x^n -\Re(c))^2 +\Im(c)^2)}{(x^n+\frac{1}{n})^4}.
\end{equation}
 Here and afterwards, $\Re(c)$ and $\Im(c)$ denote the real and imaginary part of $c$, respectively. On the real line, 
\begin{equation}\label{comp_with_x}
   C_n(x)-x=-\frac{x\left(3 n^2 x^{2 n}+n(n+5) x^n+2\right)}{2\left(n x^n+1\right)^3} \text {. }
\end{equation}
In order to determine the connectivity of the Julia set of $C_n$, we consider the two cases: $n$ is odd and $n$ is even.
First we consider the odd case.
\begin{figure}[h!]
\begin{subfigure}{.5\textwidth}
	\centering
	\includegraphics[width=1.0\linewidth]{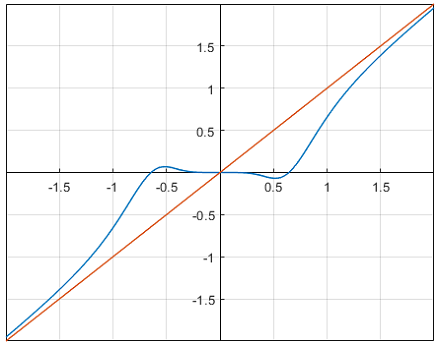}
	\caption{$n=4$}
	\label{plot_C_even}
\end{subfigure}
	\begin{subfigure}{.5\textwidth}
		\centering
		\includegraphics[width=1.0\linewidth]{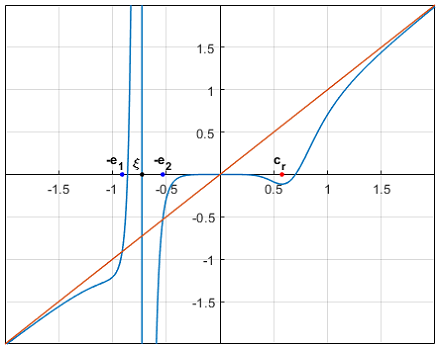}
		\caption{$n=5$}
		\label{plot_C_odd}
	\end{subfigure}
\caption{The graph of $C_n$ on the real axis  }
\end{figure}
	If $n$ is odd, then it follows from Remark~\ref{real-extra} that $C_n$ has only two real extraneous fixed points, and those are negative. Let those be $-e_1$ and $-e_2$, where $e_1>e_2>0$.
\begin{lem}
If $n$ is odd, then $(-\infty, -e_1) \subset \mathcal{A}_{\infty}$ and $\left(-e_2, 0\right) \subset \mathcal{A}_{0}$.
\label{odd-real}
\end{lem}
\begin{proof}
 The Equation (\ref{comp_with_x}) can be written as
	$$
	C_n(x)-x=-\frac{3 x\left(x^n+ e_1 ^n \right)\left(x^n+e_2 ^n\right)}{2 n\left(x^n-{\xi}^n\right)^3}, 
	$$
where $\xi$ (the $n$-th real root of $-\frac{1}{n}$) is the real pole of $C_n$. Since $\frac{-(n+5)- \sqrt{n^2+10n +1}}{6n} < -\frac{1}{n} < \frac{-(n+5) + \sqrt{n^2+10n +1}}{6n} <0$ (as $n \geq 2$), we have 
	$-e_1 < \xi < -e_2$ (see Figure \ref{plot_C_odd}  for $n=5$). Hence, 
	\begin{equation}
		C_n(x)\left\{\begin{array}{ccc}
			<x & \text { whenever } & x<-e_1 \\
			>x & \text { whenever } & -e_1<x<\xi \\
			<x & \text { whenever } & \xi<x<-e_2 \\
			>x & \text { whenever } & -e_2<x<0 \\
			<x & \text { whenever } & x>0.
		\end{array}\right.
	\end{equation}
	Recall that $C_n'(x)= \frac{x^n (x^n -c_r ^{n})((x^n -\Re(c))^2 +\Im(c))^2}{(x^n+\frac{1}{n})^4}.$
	For $x<0$, we have $x^n (x^n -c_r ^n) >0$ and therefore,
	$$
	C_n'(x)\left\{\begin{array}{ccc}
		>0 & \text { whenever } & -\infty<x<\xi \text { or } \xi<x<0 \\
		<0 & \text { whenever } & 0<x<c_r \\
		>0 & \text { whenever } & x>c_r.
	\end{array}\right.
	$$
Therefore $C_n$ is increasing in $\left(-\infty, \xi \right) \cup\left(\xi, 0\right) \cup\left(c_r, \infty\right)$ and decreasing elsewhere. It is observed that
	$$
	-\infty<-e_1<z_{-}<\xi<-e_2<0<c_r<z_{+}<\infty \text {, }
	$$ where $z_{-}$ and $z_{+}$ are the real roots of $C_n$ (see Lemma~\ref{roots} and Figure~\ref{plot_C_odd}).
	
If $I_1=(-\infty, -e_1)$ and $I_2=(-e_2, 0)$, then $C_n$ is increasing in $I_1 \cup I_2$. Also, $C_n\left(I_i\right) \subseteq I_i$ for $i=1,2$. As $C_n(x)<x$ in $I_1$, we have $\lim\limits _{k \rightarrow \infty} C_n^k(x)=-\infty$. As $C_n(x)>x$ in $I_2$, we get $\lim\limits _{k \rightarrow \infty} C_n^k(x)=0$.
	Therefore $I_1 \subset  \mathcal{A}_\infty$ and $ I_2 \subset  \mathcal{A}_0$.
\end{proof}
An immediate consequence of the above lemma is the proof of Theorem~\ref{odd-connected}.
 
\begin{proof}[Proof of Theorem~\ref{odd-connected}]
Let $c_r$ and $z_{+}$ be the real critical point and the positive root  of $C_n$, respectively. Then the image of $ \left(0, z_{+}\right) $ is contained in $ \left(C_n (c_r), 0\right) $. By assumption, $C_n (c_r)> -e_2$, and by Lemma
\ref{odd-real}, $(-e_2, 0) \subset \mathcal{A}_0$. Therefore $\left(0, z_{+}\right) \subset \mathcal{A}_0$.
\par
If for any $x > z_{+} $, $C_n^k(x)>z_{+}$ for all $k $, then $\left\{C_n^k(x)\right\}_{k>0}$ would be a strictly decreasing sequence bounded below by $z_{+}$, and hence converges to a point in $\left[z_{+}, \infty\right)$. This point must be a fixed point of $C_n$. However,  $C_n$ has no fixed point in $[z_{+}, \infty)$. Therefore, for each $x \in\left(z_{+}, \infty\right)$, there exists a natural number $k_x$ such that $0< C_n^{k_x}(x) \leq z_{+}$. It now follows from the previous paragraph that $(z_{+},\infty) \subset \mathcal{A}_0$. Thus $\mathcal{A}_0$ is unbounded. 
Therefore, $\mathcal{J}(C_n)$ is connected by Proposition \ref{Connected_J_set}.
\end{proof}

\begin{figure}[h!]
	
	\begin{subfigure}{0.5\textwidth}
		\centering
		\includegraphics[width=1.0\linewidth]{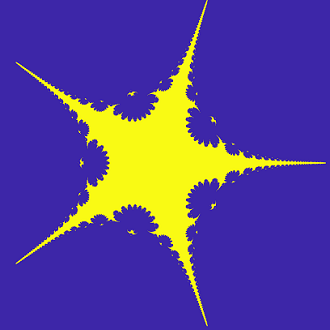}
		\caption{View in $\mathbb{C}$}
		\label{CJ_n5-origin}
	\end{subfigure}
	\begin{subfigure}{.5\textwidth}
		\centering
		\includegraphics[width=1.0\linewidth]{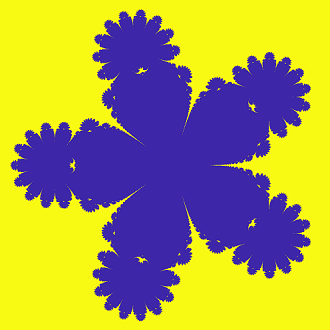}
		\caption{View from $\infty$}
		\label{CJ_n5-infinity}
	\end{subfigure}
	\caption{The Fatou and Julia set of $C_5$}
	\label{Juliaset-C_5}
\end{figure}

 The attracting basin of $0$ and the parabolic basin of $\infty$ are shown in yellow and blue, respectively, in Figure~\ref{Juliaset-C_5}. The Julia set of $C_5$ is given as the boundary of the two colors. Figure~\ref{CJ_n5-origin} shows the view from the origin, whereas  Figure~\ref{CJ_n5-infinity} is the view from infinity, showing the five attracting petals in blue.
\begin{rem}
	\begin{enumerate}
		\item It follows from the proof of Theorem~\ref{odd-connected} that $(-e_2, \infty) \subset \mathcal{A}_0$.
		\item It is numerically found that $C_n (c_r) > -e_2$ for all $n \leq 15$ but not for all $n\geq 17$ (see Table \ref{ThmC-condition}).
       
	\end{enumerate}

\end{rem}
\begin{table}[htbp]
\centering
\begin{tabular}{|c|c|c|c|c|c|}
\hline
n & $c_r$ & $C_n(c_r)$ & $-e_1$ & $-e_2$ & $C_n (c_r) > -e_2?$ \\ \hline
3 & 0.426522 & -0.031619 & -0.926694 & -0.453196 & True\\ \hline
5 & 0.576230 & -0.113824 & -0.909969 & -0.532313 & True \\ \hline
7 & 0.654932 & -0.218277 & -0.917311 & -0.590025 &  True \\ \hline
9 & 0.705328 & -0.334320 & -0.926312 & -0.633320 & True \\ \hline
11 & 0.741041 & -0.457359 & -0.934241 & -0.667096 & True \\ \hline
13 & 0.767971 & -0.585038 & -0.940878 & -0.694297 &  True \\ \hline
15 & 0.789154 & -0.715997 & -0.946404 & -0.716757 &  True \\ \hline
17 & 0.806337 & -0.849385 & -0.951038 & -0.735673 & False\\ \hline
19 & 0.820607 & -0.984639 & -0.954963 & -0.751865 & False \\ \hline
21 & 0.832680 & -1.121364 & -0.958320 & -0.765911 & False \\ \hline
23 & 0.843048 & -1.259279 & -0.961222 & -0.778232 & False \\ \hline
25 & 0.852064 & -1.398172 & -0.963752 & -0.789143 & False \\ \hline
27 & 0.859988 & -1.537883 & -0.965977 & -0.798886 & False \\ \hline
29 & 0.867015 & -1.678287 & -0.967946 & -0.807647 & False\\ \hline
31 & 0.873295 & -1.819284 & -0.969703 & -0.815576 & False \\ \hline

\end{tabular}
\caption{Critical values and extraneous fixed points of $C_n: \mathbb{R} \to \mathbb{R} $}
\label{ThmC-condition}\end{table}


Now we consider even $n$.
In this case $C_n$ is an odd function and has no real extraneous fixed point and two real critical points $-c_r$ and $c_r$. It follows from Equation (\ref{comp_with_x}) that $C_n(x)-x$ is positive if and only if $x$ is negative. In other words,
\begin{equation}
	C_n(x)	\left\{\begin{array}{ll}
	>x & \text { whenever } x<0 \\
		 <x & \text { whenever } x>0 .
	\end{array}\right.
	\label{even-n-Cf}
\end{equation}
Again from Equation (\ref{crit_C_f}) we have
$$
	C_n'(x)\left\{\begin{array}{ccc}
		>0 & \text { whenever } &-\infty<x<-c_r \\
		<0 & \text { whenever } & -c_r<x<0 \\
		<0 & \text { whenever } & 0<x<c_r \\
		>0 & \text { whenever } & x>c_r.
	\end{array}\right.
$$
Thus $C_n$ is increasing in $\left(-\infty,-c_r\right) \cup\left(c_r, \infty\right)$ and is decreasing in $\left(-c_r, c_r\right)$ (see Figure~\ref{plot_C_even} for $n=4$).
\par The proof of Theorem~\ref{even-connected} follows.
 
\begin{proof}[Proof of Theorem~\ref{even-connected}]
Let $I_{ r}=\left(0, c_r\right]$.  Then by hypothesis, for any $x\in I_r$ we have $C_n(x)>-x$. Let $C_n(x)=-y$, where $y>0$ as $C_n(x)<0$. Then we have $y<x$ and therefore, $C_n(y)>-y$. As $C_n$ is an odd function, we have $C_n^2(x)=-C_n(y)<y<x$. Thus, $C_n ^2 (I_r)$ is properly contained in $I_r$. Since $C_n$  is strictly decreasing on $I_r$, the function  $C_{n} ^2 :I_r \to I_r$ is strictly increasing. Again, since $C_n ^{2}(x) < x$ for all $x \in I_r$, the sequence $\{C_n ^{2m}(x)\}_{m>0}$ is strictly decreasing, which is also bounded below by the origin. This sequence converges, and the limit point is either a fixed point or a two-periodic point of $C_n$. However, the only such point is the origin by assumption. Therefore, $\lim\limits_{m \to \infty}C_n ^m (x) =0$ for all $x \in I_r$. In other words, $I_r \subset \mathcal{A}_0$.

\par Recall that $z_0$ denotes the positive root of $C_n$ when $n$ is even.
The images of the intervals $[c_r, z_0]$ and $[0,c_r]$ under $C_n$ are the same, and therefore $[0,z_0] \subset \mathcal{A}_0$. 

Now, if for a real number $x \in\left(z_0, \infty \right)$, $C_n^{k}(x) > z_0$ for all $k$, then $\{C_n^{k}(x)\}_{k>0} $ is a strictly decreasing sequence and is bounded below by $z_0$ and therefore converges to a fixed point of $C_n$. But there is no fixed point in $(z_0, \infty)$. Therefore, for each $x \in\left(z_0, \infty \right)$, there exists a natural number $m_x$ such that $C_n^{m_x}(x) < z_0$. Let $m_x$ be the smallest such number. Since $C_n^{m_x}(x) >0 $, we have $C_n^{m_x}(x) \in I_r$. Thus $[0,\infty] \subset \mathcal{A}_{0}$, proving that $\mathcal{A}_0$ is unbounded.
	
	Therefore, the Julia set $\mathcal{J}(C_n)$ is connected by Proposition \ref{Connected_J_set}.
\end{proof}
The proof of Corollary~\ref{cor-even} follows.
\begin{proof}[Proof of Corollary~\ref{cor-even}] 
We shall show $C_n (x) +x>0$ for all $x \in (0,1)$, and the proof will follow by applying Theorem~\ref{even-connected} since $0 < c_r <1$.
Let $y=x^n$ and $$G_n(y)=4 n^3 y^3 +9 n^2 y^2 +(7n -n^2)y +2.$$ Then $C_n (x)+x =\frac{xG(x^n)}{2(nx^n +1)^3}$ (see Equation~(\ref{formula_cheby})). Since $n$ is even, we have $x^n >0$ whenever $x>0$. Thus, we shall be done by showing $G_n(y) >0$ for all $y>0$ whenever $n \leq 16$ and is even.
	
\par  For $ n=2, 4, 6$, we have $7n -n^2>0$ and therefore $G_n(y)>0$ for all $y>0$. 
\par Let $n \geq 8$. Then $$G_n'(y)= 12 n^3 y^2 +18n^2 y -n(n-7) ~~\text{and}~~ G_n''(y)= 24 n^3 y+ 18 n^2 .$$ The positive critical point of $G_n$ is $\frac{-9 +\sqrt{12n -3}}{12n}$. Let it be denoted by $c_n$. The critical value is $G_n(c_n)=\frac{1}{24 \sqrt{3}} (3 \sqrt{3} (6n+1)-(4n-1)^{\frac{3}{2}})$ and it is a decreasing function of $n$ for $n \geq 8$. As $G_{16}(c_{16}) \approx 0.095$, we have that $G_n(c_n) >0$ for each $ n \in \{8, 10, 12, 14, 16\} $.  Since $G_n''(y) >0$ for all $n$ and $y >0$, $c_n$ is the only critical point of $G_n$ and  is the global minimum of $G_n$ in $(0, \infty)$. Therefore $G_n(y) >0$ for all even $  n \leq 16$. 
\end{proof}

Figure \ref{even-origin-view} shows the basin of attraction of $0$ for $C_4$ in yellow, which is symmetric about the origin. The blue region shows the parabolic basin of $\infty$.  Figure~\ref{even-infinity-view} shows the four attracting petals at infinity in blue.
\par 
 Here are some remarks.
\begin{rem}
\label{realcriti}
\begin{enumerate}
\item  
Under the hypothesis of Theorem~\ref{even-connected}, the whole real line is contained in the immediate basin $\mathcal{A}_0$ of $0$. This is because $C_n$ is an odd function whenever $n$ is even. Therefore, the real critical point $c_r$ is in the immediate basin of the origin. Consequently, the critical points which are solutions of $z^n =c$ and $z^n =\bar{c}$ are in $\mathcal{A}_{\infty}$.
\par
Even though $C_n (x) > -x$ for all $x \in (0, c_r)$ is not true for some even $n$, the critical point $c_r$ cannot converge to $\infty$. This is because, for each $x \in (z_0, \infty)$, we have $C_n(x) < x$ (see Equation (\ref{even-n-Cf})). Since the parabolic basin of $\infty$ must contain a critical point, this critical point is one of the (non-real) solutions of $z^n =c$ or of $z^n =\bar{c}$. However, the symmetry of the Fatou set (see Lemma \ref{Cheby_sym}) gives that all the solutions of $z^n =c$ and $z^n =\bar{c}$ are in $\mathcal{A}_{\infty}$.
\item Since $G_{18}(c_{18}) \approx -0.766$, the extension of Corollary~\ref{cor-even} is not possible beyond the value $n=18$ by the same argument used in it. However, at present, we do not know whether the Julia set of $C_n$ for  $n\geq 18$ is actually connected or not. This will be investigated in the future.
\end{enumerate}
\end{rem}

\begin{figure}[h!]
	\begin{subfigure}{0.490\textwidth}
		\centering		\includegraphics[width=1.0\linewidth]{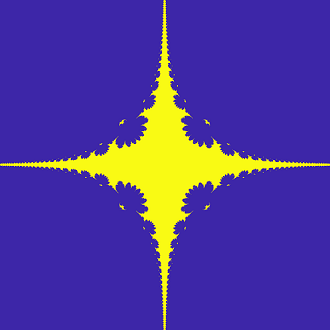}
		\caption{View in $\mathbb{C}$}
		\label{even-origin-view}
	\end{subfigure}
	\begin{subfigure}{.491\textwidth}
		\centering
		\includegraphics[width=1.0\linewidth]{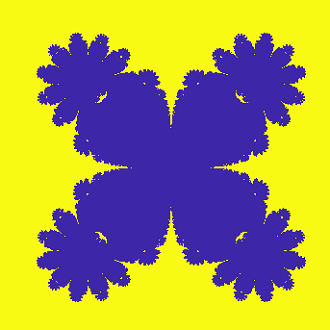}
		\caption{View from $\infty$}
		\label{even-infinity-view}
	\end{subfigure}
	\caption{The Fatou and Julia set of $C_4$}
	\label{Juliaset-C_4}
\end{figure}
We now prove the non-existence of rotation domains of $C_n$ for all even $n$.
\begin{prop}\label{noRD}
The map $C_n$ has no rotation domain whenever $n$ is even.
\end{prop}
\begin{proof}
	The non-existence of Herman ring follows from Corollary~\ref{No-HR}.
\par If there is a Siegel disk $D$, then its boundary is contained in the closure of the forward orbits of the critical points (see \cite[Theorem 11.17]{Milnor_book}). From Remark~\ref{realcriti}(1), it follows that for even $n$, the immediate basin of $\infty$ does not contain the real critical point. Hence, the only free critical points whose forward orbits need to be investigated are the solutions of $z^n=r$ for  $r>0$ as given in Lemma~\ref{criticalpoints}. Since the real line is invariant under $C_n$, the forward orbit of the real critical point $r^{\frac{1}{n}}$ is contained in the real line. As $C_n(\lambda z)=\lambda C_n(z)$ for every $\lambda$ with $\lambda^n=1$, the free critical points as well as their forward orbits are preserved under rotations of order $n$  about the origin. These forward orbits are contained in  $\bigcup_{k=0}^{n-1}\{z:z=xe^\frac{i2\pi k}{n}, x\in \mathbb{R}\}$. Since the boundary of $D$ is a closed curve contained in the Julia set, it must contain $0$ as well as $\infty$. However, $0$ is an attracting fixed point of $C_n$ leading to a contradiction. This completes the proof.
\end{proof}
\begin{rem}
    For odd $n$, if the real critical point is not in $\mathcal{A}_{\infty}$, then the critical points that are the solutions of $z^n =c$ or $z^n =\bar{c}$ are in $\mathcal{A}_{\infty}$. Thus, under this hypothesis, following the same proof of Proposition~\ref{noRD}, we conclude that $C_n$ has no rotation domain for any odd $n$.
\end{rem}
We end this section with a proposition that is valid for even as well as odd $n$.

 \begin{prop} \label{Not_inv}
 For $n\leq 16$, the immediate basin of $0$ is not completely invariant.
\end{prop}
\begin{proof}
From Remark~\ref{realcriti}(1), we get that the real critical point $c_r$ is in $\mathcal{A}_0$ for all even $n\leq 16$. The same follows from the proof of Theorem~\ref{odd-connected} for all odd $n \leq 15 $. Consequently, we get that all critical points that are solutions of $z^n=c_r^n$ are in $\mathcal{A}_0$ (Remark~\ref{cr_rot}(1)). The immediate basin $\mathcal{A}_0$ is simply connected (Proposition~\ref{Connected_J_set}). Now, it follows from the Riemann-Hurwitz formula that the degree of $C_n$ on  $\mathcal{A}_0$ is $2n +1$ for all $n \leq 16$. However, the degree of $C_n$ is $3n+1$. Thus, $\mathcal{A}_0$ is not completely invariant. 
\end{proof}

The components of the basin of $0$ different from its immediate basin are shown as yellow regions in Figures~\ref{Fatouset-C_1}, \ref{CJ_n5-origin}, and \ref{even-origin-view}.
\section{A remark on Newton's method}
\label{Newton}
 The point at $\infty$ is the only weakly repelling fixed point of $N_{pe^q}$ for every pair of polynomials $p$ and $q$. It follows from \cite[Theorem I]{Shishikura2009} that $\mathcal{J}(N_{pe^q})$ is connected. However,  detailed  dynamics such as the nature of Fatou components and the symmetry group of the Julia set have not been explicitly worked out for any specific $p, q$. We do it for  $p(z)=z$ and $q(z)=z^n$.
\begin{prop}\label{New_sym}
	If $f(z)=z e^{z^n}$ for $n \geq 1$, then the symmetry group of its Newton's method $N_f$ is  $\{z \mapsto \lambda z: \lambda ^n =1\}$. Then $0$ and $\infty$ are attracting and parabolic fixed points of $N_f$ respectively. Further, the Fatou set of $N_f$ is the union of the basins of $0$ and $\infty$, and the immediate basin of $0$ is completely invariant. 
	\label{newton-map}\end{prop}
\begin{proof}
	Note that 
	\begin{equation}\label{nwt}
		N_f(z)=\frac{n z^{n+1}}{n z^n+1} ~~\text{and}~~N_f = \phi \circ P \circ \phi^{-1},\end{equation}
	where $P(z)=z(\frac{1}{n} z^n +1)$ and $\phi(z)=\frac{1}{z}$. The symmetry group of $P$ is  $\{z \mapsto \lambda z: \lambda ^n =1\}$, and the Julia set of $N_f$ is  $\phi(\mathcal{J}(P))$ (\cite[Theorem 3.1.4]{Beardon_book}). The composition of each rotation about the origin with $\phi$ remains a rotation about the origin under, and $\phi =\phi^{-1}$. Therefore, for every $\sigma \in \Sigma P,$ we have
	$\sigma(\mathcal{J}(P))=\mathcal{J}(P)$, and consequently $\sigma \phi\left(\mathcal{J}\left(N_f\right)\right)=\mathcal{J}(P)=\phi\left(\mathcal{J}\left(N_f\right)\right).$ In other words, 
	$ \phi^{-1} \circ \sigma \circ \phi$ maps the Julia set of $N_f$ onto itself. As this is a Euclidean isometry, we have $ \phi^{-1} \circ \sigma \circ \phi \in \Sigma N_f$. Conversely, for each $\tilde{\sigma} \in \Sigma N_f$, the M\"{o}bius map  $\phi^{-1} \circ \tilde{\sigma} \circ \phi $ takes the Julia set of $P$ onto itself and is a Euclidean isometry. Thus, $\Sigma N_f =\{\phi \circ \sigma \circ \phi^{-1}: \sigma \in \Sigma P\}$=$\{z \mapsto \lambda z: \lambda ^n =1\}$.  
	\par The polynomial $P$ has a parabolic fixed point at the origin with multiplicity $n+1$ and has a superattracting fixed point at $\infty$. There are $n$ critical points, namely the solutions of $z^n =-\frac{n}{n+1}$, each of which is contained in an attracting petal of $0$. As the immediate basin of infinity of a non-linear polynomial is completely invariant, we have that $\mathcal{A}_{P,\infty}$ is completely invariant. Therefore, the Fatou set of $P$ is the union of the basin of attraction $\mathcal{B}_{P,0}$ and the immediate basin $\mathcal{A}_{P,\infty}$ corresponding to $0$ and $\infty$, respectively.
		As $N_f =\phi \circ P\circ \phi^{-1}$, the points $0$ and $\infty$ are superattracting and parabolic fixed points of $N_f$, respectively. Further, the Fatou set of $N_f$ is the union of the immediate basin $\mathcal{A}_{N_f,0}$ of $0$ and the basin $\mathcal{B}_{N_f,\infty}$ of $\infty$. It is clear that $\mathcal{A}_{N_f,0}$ is completely invariant.

	The proof is complete.
\end{proof}
The Fatou set of  $N_f$, where $f(z)=ze^{z^3}$, is the union of the immediate basin of $0$ and the basin of $\infty$, are shown in yellow and blue, respectively (see Figure \ref{newton-3}). Note that $\mathcal{A}_\infty$ consists of the three bigger blue lobes (not the smaller one) and is not completely invariant.
\begin{figure}[h!]
	\begin{subfigure}{.48\textwidth}
		\centering
		\includegraphics[width=0.98\linewidth]{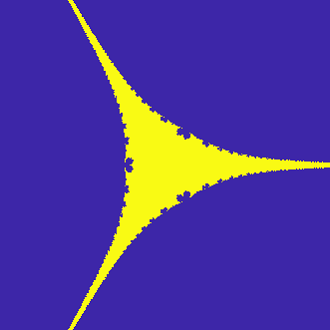}
		\caption{View in $\mathbb{C}$}
	\end{subfigure}
	\begin{subfigure}{.48\textwidth}
		\centering
\includegraphics[width=0.98\linewidth]{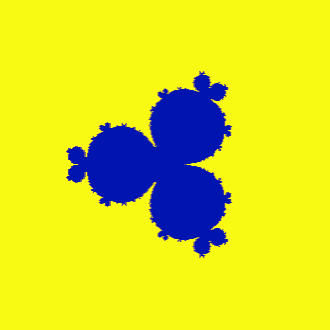}
		\caption{View from $\infty$}
	\end{subfigure}
	\caption{The Fatou set of the  Newton's method applied to $ze^{z^3}$
	}
	\label{newton-3}
\end{figure}

\textbf{Acknowledgment:} The authors are very much thankful to Prof. S. Ponnusamy and Prof. G. Liu for suggesting a proof of Theorem~\ref{Characterization}. This work was initiated when the author, Soumen Pal, was associated with the project (Grant No. CRG/2022/003560 (SERB)) at Indian Institute of Science Education and Research Kolkata. He is currently supported by Indian Institute of Technology
Madras through a Postdoctoral Fellowship. Pooja Phogat, is supported by a Senior Research Fellowship (Grant No. 09/1059(0031)/2020-EMR-I) provided by Council of Scientific and Industrial Research, Govt. of India.


\begin{thebibliography}{00}
\bibitem{Beardon_book} Beardon, A.F., Iteration of Rational Functions, Graduate Texts in Mathematics, Vol. 132, Springer-Verlag, New York, 1989.	

\bibitem{Ber1993} Bergweiler, W., Iteration of meromorphic functions, Bull. Amer. Math. Soc., 29 (1993), no. 2, 151--188.

\bibitem{CCV2020} Campos, B., Canela, J., Vindel, P., Connectivity of the Julia set for the Chebyshev-Halley family on degree $n$ polynomials, Commun. Nonlinear Sci. Numer. Simul., 82 (2020), p. 105026.

\bibitem{Cayley1879} Cayley, A., Applications of the Newton-Fourier method to an imaginary root of an equation, J. of Pure and App. Math., 16 (1879), 179--185.
       
\bibitem{CMHD2022} Cumsille, P., Gonz\'{a}lez-Mar\'{i}n, J., Honorato, G., Lugo, D., Disconnected Julia set of Halley's method for exponential maps, Dyn. Syst., 37 (2022), no. 2, 280--294.
		
\bibitem{DS2022} Drach, K., Schleicher, D., Rigidity of Newton dynamics,
Adv. Math., 408 (2022), part A, Paper No. 108591, 64 pp.

\bibitem{GGM2015} García-Olivo, M., Gutiérrez, J.M., Magreñán, Á.A., A complex dynamical approach of Chebyshev’s method, SeMA J., 71 (2015), 57--68.
		
\bibitem{GV2020} Gutiérrez, J.M., Varona J.L., Superattracting extraneous fixed points and $n$-cycles for Chebyshev’s method on cubic polynomials, Qual. Theory Dyn. Syst., 19 (2020), no. 2, 23.

\bibitem{Haruta1999} Haruta, M.E., Newton’s method on the complex exponential function, Trans. Amer. Math. Soc., 351 (1999), no.6, 2499--2513.

\bibitem{HSS2001} Hubbard, J., Schleicher, D., Sutherland, S., How to find all roots of complex polynomials
by Newton's method, Invent. Math., 146 (2001), no. 1, 1--33.

\bibitem{lai1993} Laine, I., Nevanlinna Theory and Complex Differential Equations, De Gruyter Stud. Math., 15
Walter de Gruyter  Co., Berlin, 1993. 
	
\bibitem{Lei1997} Lei, T., Branched coverings and cubic Newton maps, Fund. Math., 154 (1997), no. 3, 207--260.

\bibitem{Mamayusupov2019} Mamayusupov, K., A characterization of postcritically minimal Newton maps of complex exponential functions, Ergodic Theory Dynam. Systems, 39 (2019), 2855-2880.

\bibitem{Mayer_Schleicher2006}  Mayer, S.,  Schleicher, D., Immediate and virtual basins of Newton’s method for entire functions,
Ann. Inst. Fourier (Grenoble), 56 (2006), 325--336.

\bibitem{Milnor_book} Milnor, J., Dynamics in One Complex Variable: Introductory Lectures, 3rd ed., Princeton University Press, New Jersey, 2006.

\bibitem{Nayak-Pal2022} Nayak, T., Pal, S., The Julia sets of Chebyshev's method with small degrees, Nonlinear Dyn., 110 (2022), no. 1, 803--819.

\bibitem{Sym_dym} Nayak, T., Pal, S., Symmetry and dynamics of Chebyshev’s method, Mediterr. J.
Math. 22 (2025), no. 1, Paper No. 12.

\bibitem{Felex1989}	Przytycki, F., Remarks on the simple connectedness of basins of sinks for iterations of rational maps, Ergodic Theory Dynam. Systems, 23 (1989), 229--235.

\bibitem{RC2007} R\"{u}ckert, J., Schleicher, D., On Newton’s method for entire functions, J. Lond. Math. Soc., 75 (2007), no. 3, 659--676.
 
\bibitem{Shishikura2009} Shishikura, M., The connectivity of the Julia set and fixed points, Complex Dynamics, A K Peters, Wellesley, MA, 2009, 257--276.
       
\bibitem{yy2006} Yang, C., Yi, H., Uniqueness theory of meromorphic functions, Beijing, Science Press, 2006.
 

\end{thebibliography}
\end{document}